%% file: VersionarXiv.tex
\RequirePackage{rotating}
\documentclass[10pt,fleqn]{amsart}
\usepackage{amsmath,bm}
\usepackage{amsthm}

\usepackage[colorlinks,hyperindex]{hyperref}

\usepackage{graphicx}
\usepackage{picins}

\makeatletter
\newcommand{\addresseshere}{%
  \enddoc@text\let\enddoc@text\relax
}
\makeatother

\usepackage{amsfonts}
\usepackage{latexsym}
\usepackage{amssymb}
\usepackage{graphicx}
\usepackage{picins}

\include{MyShortcuts}
\usepackage{natbib}

\begin{document}
\title[Discrete FDR step-up procedures]{A discrete modification of the Benjamini-Yekutieli procedure}
\author{Sebastian D\"ohler\\
Darmstadt University of Applied Sciences}
\address{Sebastian D\"ohler\\ University of Applied Sciences\\
CCSOR and Faculty of Mathematics and Science\\
D-64295 Darmstadt\\
Germany}
\email{sebastian.doehler@h-da.de}
\date{\today}
\begin{abstract}
The Benjamini-Yekutieli procedure is a multiple testing method that controls the  false discovery rate under arbitrary dependence of the $p$-values. A modification of this and related procedures is proposed for the case when the test statistics are discrete. It is shown that taking discreteness into account can improve upon known procedures. The performance of this new procedure is evaluated for pharmacovigilance data and in a simulation study.
\end{abstract}
\maketitle
\input{Introduction2}

\input{TheoreticalResults}
\input{Example}
\input{Simulation}

\input{Discussion}
\bibliographystyle{chicago}
\label{bibliography}
\bibliography{Vorlage}
\addresseshere
\appendix
\newpage
\input{Appendix}
\end{document}

%% file: MyShortcuts.tex
\usepackage{url}		
\usepackage{xspace}		
\usepackage{booktabs}

\usepackage{rotating}	
\usepackage{pdflscape}
\usepackage{fancyvrb}

\usepackage{tabularx}  

\usepackage{bbm}			

\newtheorem{assumption}{Assumption}

\newtheorem{prop}{Proposition}
\newtheorem{coro}{Corollary}

\newtheoremstyle{break}
  {9pt}
  {9pt}
  {}
  {}
  {\bfseries}
  {}
  {\newline}
  {}
\newcommand{\HG}{{[HG]}\xspace}
\newcommand{\midPV}{{mid-$p$-values}\xspace}
\newcommand{\Gunif}{{G_{\textnormal{unif}}}}

\graphicspath{{../..//}}

\newcommand{\indicator}{\mathbbm{1}}

\newcommand{\erw}{\textnormal{E}}

\newcommand{\FDR}{\textnormal{FDR}}
\newcommand{\FDP}{\textnormal{FDP}}

\newcommand{\FWER}{\textnormal{FWER}}
\newcommand{\kFDP}{\textnormal{kFDP}}
\newcommand{\kFDR}{\textnormal{kFDR}}



%% file: Introduction2.tex
\section{Introduction}
Consider the problem of testing $n$ hypotheses $H_1, \ldots, H_n$ simultaneously. A classical approach to dealing with the multiplicity problem is to control the familywise error rate (FWER), i.e. the probability of one or more false rejections. However, when the number $n$ of hypotheses is large, the ability to reject false hypotheses is small. Therefore, alternative Type 1 error rates have been proposed that relax control of FWER in order to reject more false hypothesis. One such errror rate is the expected proportion of false rejections amongst all rejections, which is known as the false discovery rate (FDR). \citet{BenjaminiHochberg95} showed that the so-called linear step-up procedure controls the $\FDR$ under independence, \citet{BenjaminiYekutieli01} proved that it also controls the FDR for certain types of positive dependence of the p-values. The latter publication also introduced a rescaled version of the linear step up procedure - referred to as the Benjamini-Yekutieli (BY) procedure hereafter - which guarantees FDR control under any type of dependence among the p-values. In this paper we focus on the Benjamini-Yekutieli procedure.

Generally, the only available information on the distribution of $p$-values is that it is stochastically larger than the uniform distribution. However, for discrete tests, there are important practical situations in which the (conditional) $p$-value distributions under the null hypotheses are in fact known. Examples of such tests include Fisher's exact test (see e.g. \citet{WestWolf1997} for an application to clinical studies), the McNemar test (see e.g. \citet{WestfTroendlePenello} for an application to classification models) and the binomial test (see e.g. \citet{DoergeChen2015} for an application to next generation sequencing data).  When the $p$-value distributions are available it may be possible to improve classical multiple testing procedures by incorporating discreteness of the distribution functions into the multiplicity adjustment while still controlling the Type 1 error rate. For the FWER, such methods were obtained e.g. by \citet{WestWolf1997} and \citet{HommKrumm1998}. For a review of such methods, see \citet{Gutman2007}. For the FDR, \citet{Heller2012}, in the following abbreviated as \HG, give an overview of existing approaches for taking discreteness into account. Moreover, they modify the Benjamini-Liu (BL) step-down procedure \citep{BenjaminiLiu1999} and prove FDR control for the case when the discrete test statistics are independent as well as for a certain type of dependency. Recently, \citet{DoergeChen2015} have used grouping and weighting approaches for independent test statistics. They point out that 'How to derive better FDR procedures in the discrete paradigm remains an urgent but still unresolved problem.' This paper is a contribution in this direction, for the case where the $p$-values are arbitrarily dependent.

In this paper we develop step-up procedures that control the FDR under arbitrary dependence of the $p$-values and improve upon known procedures when the test statistics are discrete. Our approach is similar in spirit to the one described by \citet{WestWolf1997}. The paper is organised as follows. First we present some theoretical results, building on the work of \citet{Sarkar2008}. Then we compare the performance with some of the procedures considered by \HG by analysing a pharmacovigilance data set provided by \HG in the R package \texttt{discreteMTP} and in a simulation study. The paper concludes with a discussion.

%% file: TheoreticalResults.tex
\section{Theoretical results} \label{sec:TheoreticalResults}
When testing hypotheses $H_1, \ldots, H_n$ we assume that corresponding $p$-values $PV_1, \ldots, PV_n \in [0,1]$ with distribution functions $F_1, \ldots, F_n$ under the null hypotheses are available. For $x \in [0,1]$ define the functions $G(x)= F_1(x) + \cdots + F_n(x)$ and $\Gunif(x)=n \cdot x$. Following \citet{LehmannRomano2005}, the only distributional assumption we make in this paper is the following.
\begin{assumption}
\label{ass:pvalues}
For any true hypothesis $H_i$ we assume that the distribution of the $p$-value $PV_i$ is stochastically larger than a uniform random variable, i.e. $F_i(u)\le u$ for all $u \in (0,1)$. 
\end{assumption}

Let $\mathcal{C}=\{c \in [0,1]^n | c_1 \le \cdots \le c_n\}$ denote the set of non-decreasing critical values. For $c \in \mathcal{C}$ the corresponding step-up procedure rejects hypotheses $H_{(1)}, \ldots, H_{(k)}$, where $k= \max \{i| PV_{(i)} \le c_i\}$. If no such $i$ exists, no hypothesis is rejected. For the corresponding step-down procedure, reject $H_{(1)}, \ldots, H_{(k)}$, where $k= \max \{i| PV_{(j)} \le c_j, \quad j=1, \ldots, i\}$. If $PV_{(1)}>c_1$, no hypothesis is rejected.

Our main tool for proving $\FDR$ control for discrete step-up procedures is an upper bound of the FDR in terms of the function $G$.
\begin{prop} \label{prop:MainProp}
Let $PV_1, \ldots, PV_n$ satisfy assumption \ref{ass:pvalues}. For any sequence $c \in \mathcal{C}$ we have for the step-up procedure (with $c_0=0$)
\begin{align}
\FDR(c) & \le  \sum_{r=1}^n \sum_{j \in I} \frac{1}{r} (F_j(c_r)-F_j(c_{r-1}))  \label{eq:SarkarBound1}\\
&\le \sum_{r=1}^n \frac{1}{r} (G(c_r)-G(c_{r-1})) \label{eq:SarkarBound2}
\end{align}
where $I \subset \{1, \ldots, n\}$ denotes the set of true null hypotheses.
\end{prop}

\subsubsection*{Comments} If $PV_i\sim U(0,1)$ for $i \in I$, the bound \eqref{eq:SarkarBound1} yields the statement of Theorem 3.5 in \citet{Sarkar2008}. 

We have introduced assumption \ref{ass:pvalues} mainly because it is commonly considered a desirable property of $p$-values. However, as the following proof of the proposition shows, we actually only need that $F_j(0)=0$ and so this weaker assumption is sufficient for the proposition to hold. For discrete tests this is relevant because there are usually different ways of defining $p$-values, see \citet{Hirji06} for a comprehensive review of such possibilities. For instance, \HG consider in their analysis both exact and \midPV, which may not satisfy assumption \ref{ass:pvalues} but the weaker assumption $F_j(0)=0$. For \midPV this implies that the proposition remains true with corresponding function $G$ as long as $P(mid-PV_i=0)=0$. Thus the results presented here carry over naturally and flexibly to \midPV (and other variants) but since we are mainly interested in exact $p$-value methods we do not pursue this further.


\begin{proof}
From \citet{Sarkar2007}, Lemma 3.3, we obtain (with $k=1$ and $\mathcal{C}_1^0=I$)
\begin{align*}
\FDR &\le \sum_{j \in I} P(PV_j \le c_1) + \sum_{r=2}^n \sum_{j \in I} \erw \left[ P(R_{n-1}^{(-j)}) \ge r-1 | PV_j) \cdot \left( \frac{\indicator_{\{PV_j \le c_r\}}}{r} - \frac{\indicator_{\{PV_j \le c_{r-1}\}}}{r-1}\right)\right]
\intertext{where $R_{n-1}^{(-j)}$ is the number of rejections for the step-up procedure based on $\{PV_i | i \in I \setminus \{j\}\}$ and the critical values $c_2 \le \cdots \le c_n$. Since $P(R_{n-1}^{(-j)}) \ge r-1 | PV_j) \in [0,1]$ and }
&\frac{\indicator_{\{PV_j \le c_r\}}}{r} - \frac{\indicator_{\{PV_j \le c_{r-1}\}}}{r-1} \le  \frac{1}{r}(\indicator_{\{PV_j \le c_{r}\}}-\indicator_{\{PV_j \le c_{r-1}\}})
\intertext{we obtain by taking expectations}
\FDR &\le \sum_{j \in I} F_j(c_1) + \sum_{r=2}^n \sum_{j \in I} \frac{1}{r} (F_j(c_r)-F_j(c_{r-1})).
\end{align*}
Since $F_j(c_0)=0$ by assumption \ref{ass:pvalues}, the result follows.
\end{proof}
Under assumption \ref{ass:pvalues}, \citet{BenjaminiYekutieli01} showed that the step-up procedure (in the following abbreviated as BY) based on the critical constants 
\begin{align}
c_i &= \frac{i}{n\cdot D_{BY}} \alpha,\text{ where }  D_{BY}= \sum_{i=1}^n \frac{1}{i} \label{eq:DefBYProcedure}
\intertext{controls the FDR at level $\alpha$, while \citet{Sarkar2008b} obtained a similar result for}
c_i &= \frac{i(i+1)}{2n^2} \alpha. \label{eq:DefSarkarProcedure}
\end{align}

From the proposition we obtain a simple way of constructing step-up procedures that control the FDR. The following result is partly a generalisation of Theorem 3.5 of \citet{Sarkar2008}.
\begin{coro} \label{coro:SpecialProcedures}
Let $PV_1, \ldots, PV_n$ satisfy assumption \ref{ass:pvalues}. For any sequence $0=\widetilde{y}_0 \le \widetilde{y}_1 \le \cdots \le \widetilde{y}_n$ define 
\begin{align}
D &= \sum_{i=1}^n \frac{1}{i} (\widetilde{y}_i-\widetilde{y}_{i-1}) \nonumber 
\intertext{and for $\alpha \in (0,1)$ let $c \in \mathcal{C}$ satisfy}
G(c_i) &\le y_i:=\frac{\alpha}{D}\widetilde{y}_i \qquad \text{for $i=1, \ldots, n$.} \label{eq:DefCritConstResc}
\end{align}
\begin{itemize}
	\item[(a)] Then the step-up procedure with critical constants $c$ controls the FDR at level $\alpha$ under arbitrary dependence of the $p$-values.
\item[(b)] For the choice $\widetilde{y}_i=i$ we obtain a modified Benjamini-Yekutieli procedure.
\item[(c)] For the choice $\widetilde{y}_i=i(i+1)$ we obtain a modification of Sarkar's procedure.
\end{itemize}
\end{coro}
\begin{proof} For statement (a) we have by \eqref{eq:SarkarBound2}
\begin{align*}
\FDR(c) & \le \sum_{i = 1}^n  \frac{1}{i} (G(c_i)-G(c_{i-1}))= \sum_{i = 1}^{n-1}  \frac{1}{i(i+1)} G(c_i) + \frac{1}{n}G(c_n) \\
&\le \sum_{i = 1}^{n-1}  \frac{1}{i(i+1)} y_i + \frac{1}{n}y_n = \sum_{i = 1}^{n} \frac{1}{i} (y_i-y_{i-1})\\
&= \frac{\alpha}{D} \sum_{i = 1}^{n} \frac{1}{i}(\widetilde{y}_i-\widetilde{y}_{i-1})=\alpha
\end{align*}
where the second inequality follows from \eqref{eq:DefCritConstResc}.

For statement (b) we have 
\begin{align*}
D&= \sum_{i = 1}^n \frac{1}{i}(i-(i-1))=\sum_{i = 1}^n \frac{1}{i}= D_{BY}.
\end{align*}
If $G=\Gunif$ we recover the Benjamini-Yekutieli given by \eqref{eq:DefBYProcedure}.

For statement (c) we have 
\begin{align*}
D&= \sum_{i = 1}^n \frac{1}{i}(i(i+1)-(i-1)i)=2n
\end{align*}
which is the rescaling constant from Sarkar's procedure. If $G=\Gunif$ we recover Sarkar's constants from \eqref{eq:DefSarkarProcedure}.
\end{proof}
\subsubsection*{Comments} The corollary presents a generic method of constructing FDR controlling step-up procedures under general dependency. These procedures basically work by choosing values $y_1, \ldots,y_n$ and comparing these to the values of the function $G$ for the observed $p$-values and for arbitrary distribution of the $p$-values under the null hypotheses. More specifically, the step-up procedure rejects hypotheses $H_{(1)}, \ldots, H_{(k)}$, where $k= \max \{i| G(PV_{(i)}) \le y_i\}$. Adjusted $p$-values for the procedure defined by \eqref{eq:DefCritConstResc} are obtained using e.g. Procedure 1.4 from \citet{DudoitLaan2007}:
\begin{align}
\widetilde{pv}_{(i)} &= \min_{j=i, \ldots,n} \left\{ \min \left( \frac{D}{j}G(pv_{(j)}),1\right)\right\}. \label{eq:Def:Adjusted:pvalues}
\end{align}

When $PV_1, \ldots, PV_n \sim U(0,1)$ we have $G(x)=\Gunif(x)=n\cdot x$ and we obtain $c_i=\frac{1}{n}y_i$ which yields the Benjamini-Yekutieli constants for $\widetilde{y}_i=i$ and the procedure of \citet{Sarkar2008} for $\widetilde{y}_i=i(i+1)$. When the $p$-values are distributed discretely, the function $G(x)= F_1(x) + \cdots + F_n(x)$ may be considerably smaller than $\Gunif$, as explained in \HG. Thus, procedures based on $\Gunif$ rather than on $G$ may be substantially (and needlessly) more conservative. We present an example below. As \citet{WestWolf1997} point out in the context of FWER control, the degree of improvement 'depends upon the specific characteristics of the discrete distributions. Larger gains are possible when the number of tests is large and where many variables are sparse'.

For independent statistics, \citet{Heyse2011} introduces a modification of the BH procedure (DBH) by defining critical constants $c_i$ by $G(c_i)\le i \cdot \alpha$. This corresponds to our modified BY procedure where $D$ is set to 1. Since $D_{BY} = \log(m) + \gamma$ with $\gamma\approx 0.57721$ the Euler–-Mascheroni constant, this procedure is considerably more powerful than the discrete BY procedure. However, as we discuss in the appendix, the DBH procedure does not generally control the FDR at the desired level.

\subsubsection*{Example} Figure \ref{fig:CompareDiscGContG} presents the graphs of $G$ and $\Gunif$ for the pharmacovigilance data from \HG which is described in more detail in the next section. In this case, $n=2446$. For $x \in (0,0.001)$ the shape of $G$ is roughly linear and a numerical comparison yields $\frac{G}{\Gunif} \approx \frac{1}{3.8}$. This means that the classical step-up critical constants based on $\Gunif$ like the BY or Sarkar procedure can be multiplied by a factor of roughly $3.8$ and still maintain FDR control in this special discrete setting.
%
%

\begin{figure}[htbp]
	\centering
		\includegraphics[width=1.00\textwidth]{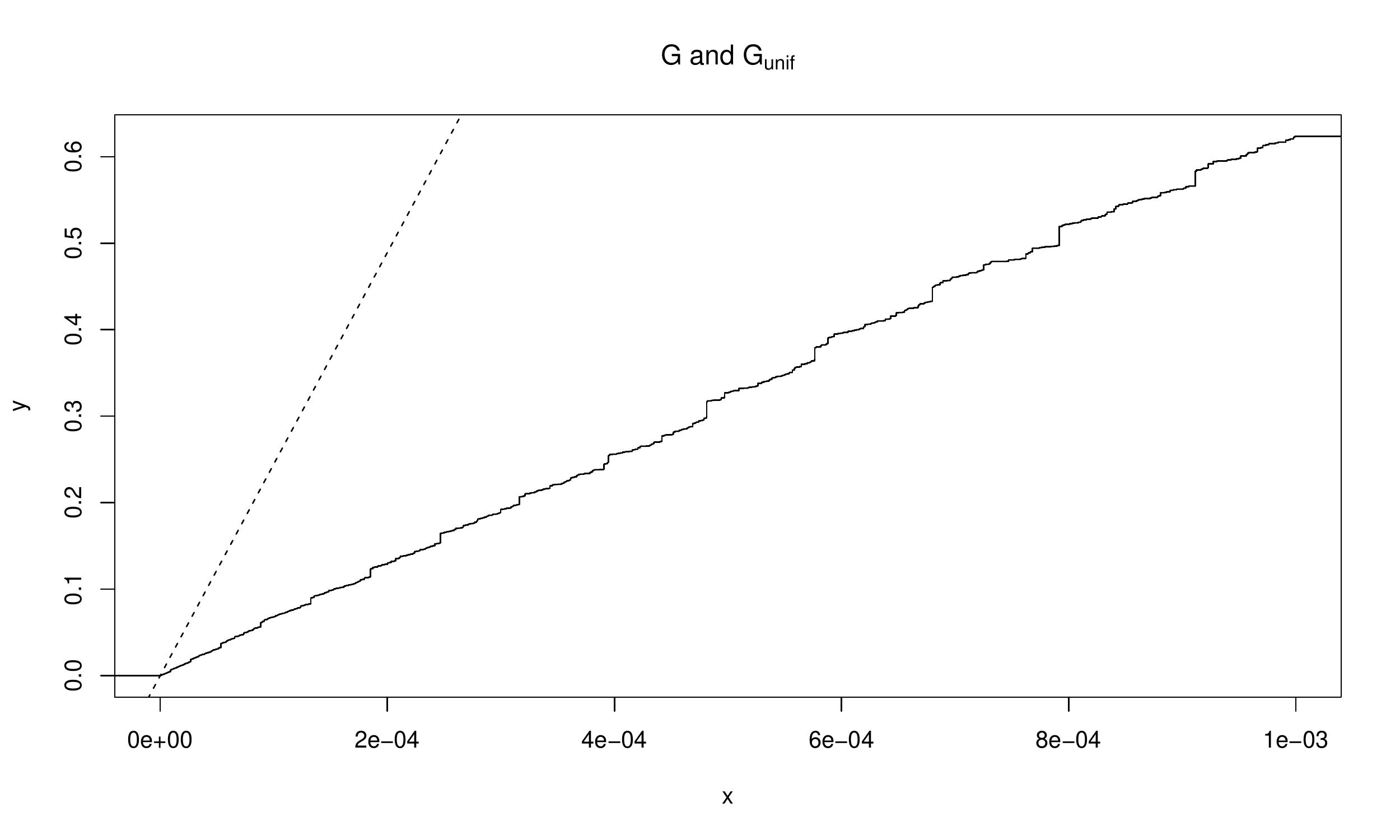}
	\caption{Graphs of $G$ (solid line) and $\Gunif$ (dotted line) for the pharmacovigilance data from \HG}
	\label{fig:CompareDiscGContG}
\end{figure} 

%% file: Example.tex
\section{Empirical data}
We revisit the pharmacovigilance data analysed by \HG, to which we also refer for more details. This data is derived from a database for reporting, investigating and monitoring adverse drug reactions due to the Medicines and Healthcare products Regulatory Agency in the United Kingdom. It contains the number of reported cases of amnesia as well as the total number of adverse events reported for each of the 2446 drugs in the database. Overall, there are 2051 cases of amnesia. \citet{Heller2012} investigate the association between reports of amnesia and suspected drugs by performing for each drug Fisher's exact test (one-sided) for testing for association between the drug and amnesia and adjusting for multiplicity by using several (discrete) FDR procedures.

The primary goal of our analysis is to compare the (continuous) Benjamini-Yekutieli and Sarkar procedures to the discrete modifications introduced in section \ref{sec:TheoreticalResults}. All these methods guarantee FDR control under arbitrary dependence of the $p$-values. As a secondary analysis we also present results  for the (continuous) Benjamini-Hochberg and Benjamini-Liu procedures as well as the discrete Benjamini-Liu procedure developed by \HG and the discrete BH procedure due to \citet{Heyse2011}.

Table \ref{tab-Example} presents the results of our analysis. For the BY and Sarkar procedures the adjusted $p$-values are reduced roughly by a factor of 3--4 by the discrete modifications as expected. Except for the Sarkar procedure all discrete modifications lead to some additional rejections. While the BH, BL and DBL procedures control FDR only for certain types of dependence, the results in the first four columns of the table are valid under any type of dependence.  The discrete BY not only rejects the same hypotheses as the discrete BL procedure, but, as a comparison of columns 2 and 8 shows, its adjusted $p$-values are only half as large. For instance, while the  adjusted $p$-value for association between Fluoxetine and amnesia is $12.4\%$ for the discrete BL, it barely fails to be significant at the $5\%$-level for the discrete BY. Thus, at least in this setting, the discrete BY procedure, while requiring less restrictive dependence assumptions, may be more powerful than the BL procedure. As expected, the DBH procedure rejects the most hypothesis, followed by the BH procedure. While the focus of this paper is on FDR procedures for arbitrarily dependent data, we would like to follow one of the reviewers who noted that for safety data, treating the data as independent may be preferable. Thus, it is not our intention of suggesting that the (discrete) BY procedure would be the ideal method for safety analyses. Rather, the rationale for analysing this specific data set was to facilitate a comparison to the work of HG.

\begin{table}[ht]
\caption{The 27 smallest adjusted $p$-values from the multiple
testing procedures for the pharmacovigilance data. The columns from left to right are the
adjusted $p$-values from the (continuous and discrete) Benjamini-Yekutieli procedure, the (continuous and discrete) Sarkar procedure, the (continuous and discrete) BH procedure and the (continuous and discrete) BL procedure.}\label{tab-Example}
\begin{center}
\scriptsize
\begin{tabular}{lcccccccc}
\toprule
 & BY &  DBY & Sarkar& DSarkar						& BH  		&  DBH					& BL & DBL\\ \midrule
	BUPROPION & 0.4161 & 0.1112 & 1.0000 & 1.0000 & 0.0497 & 0.0133 & 0.6864 & 0.2659 \\
    CITALOPRAM & 0.0080 & 0.0018 & 0.2902 & 0.0661 & 0.0009 & 0.0002 & 0.0140 & 0.0032 \\
    DEXAMPHETAMINE & 0.0606 & 0.0169 & 1.0000 & 0.4474 & 0.0072 & 0.0020 & 0.1386 & 0.0403 \\
    ETHANOL & 1.0000 & 0.2709 & 1.0000 & 1.0000 & 0.1263 & 0.0323 & 0.9515 & 0.5541 \\
    FLUOXETINE & 0.1956 & 0.0518 & 1.0000 & 1.0000 & 0.0233 & 0.0062 & 0.3956 & 0.1239 \\
    GABAPENTIN & 0.0000 & 0.0000 & 0.0000 & 0.0000 & 0.0000 & 0.0000 & 0.0000 & 0.0000 \\
    INDOMETHACIN & 0.0000 & 0.0000 & 0.0000 & 0.0000 & 0.0000 & 0.0000 & 0.0000 & 0.0000 \\
    LACOSAMIDE & 0.2752 & 0.0710 & 1.0000 & 1.0000 & 0.0328 & 0.0085 & 0.5222 & 0.1721 \\
    LEVETIRACETAM & 0.0446 & 0.0117 & 1.0000 & 0.3408 & 0.0053 & 0.0014 & 0.0948 & 0.0255 \\
    LITHIUM & 0.0012 & 0.0003 & 0.0464 & 0.0105 & 0.0001 & 0.0000 & 0.0020 & 0.0004 \\
    LORAZEPAM & 0.0000 & 0.0000 & 0.0000 & 0.0000 & 0.0000 & 0.0000 & 0.0000 & 0.0000 \\
    MEFLOQUINE & 0.0000 & 0.0000 & 0.0000 & 0.0000 & 0.0000 & 0.0000 & 0.0000 & 0.0000 \\
    MIDAZOLAM & 0.0189 & 0.0050 & 0.6480 & 0.1723 & 0.0023 & 0.0006 & 0.0350 & 0.0094 \\
    OXCARBAZEPINE & 1.0000 & 0.2912 & 1.0000 & 1.0000 & 0.1364 & 0.0348 & 0.9636 & 0.5931 \\
    PAROXETINE & 0.0000 & 0.0000 & 0.0000 & 0.0000 & 0.0000 & 0.0000 & 0.0000 & 0.0000 \\
    PREGABALIN & 0.0000 & 0.0000 & 0.0000 & 0.0000 & 0.0000 & 0.0000 & 0.0000 & 0.0000 \\
    RIMONABANT & 0.0000 & 0.0000 & 0.0000 & 0.0000 & 0.0000 & 0.0000 & 0.0000 & 0.0000 \\
    SERTRALINE & 0.5762 & 0.1547 & 1.0000 & 1.0000 & 0.0688 & 0.0185 & 0.8098 & 0.3599 \\
    SIMVASTATIN & 0.0000 & 0.0000 & 0.0000 & 0.0000 & 0.0000 & 0.0000 & 0.0000 & 0.0000 \\
    STRONTIUM\_RANELATE & 0.0550 & 0.0156 & 1.0000 & 0.4340 & 0.0066 & 0.0019 & 0.1211 & 0.0357 \\
    TEMAZEPAM & 0.0007 & 0.0002 & 0.0312 & 0.0072 & 0.0001 & 0.0000 & 0.0011 & 0.0003 \\
    TOPIRAMATE & 0.0380 & 0.0104 & 1.0000 & 0.3373 & 0.0045 & 0.0012 & 0.0732 & 0.0205 \\
    TRIAZOLAM & 0.0000 & 0.0000 & 0.0000 & 0.0000 & 0.0000 & 0.0000 & 0.0000 & 0.0000 \\
    VARENICLINE & 0.0000 & 0.0000 & 0.0000 & 0.0000 & 0.0000 & 0.0000 & 0.0000 & 0.0000 \\
    VIGABATRIN & 0.0418 & 0.0113 & 1.0000 & 0.3408 & 0.0050 & 0.0013 & 0.0846 & 0.0234 \\
    ZOLPIDEM & 0.0000 & 0.0000 & 0.0001 & 0.0000 & 0.0000 & 0.0000 & 0.0000 & 0.0000 \\
    ZOPICLONE & 0.0000 & 0.0000 & 0.0000 & 0.0000 & 0.0000 & 0.0000 & 0.0000 & 0.0000 \\
    \bottomrule
    
  Number rejections	&19 &  21 																			 &14							& 14													& 24					& 27	    &		16 & 21\\\bottomrule
\end{tabular}
\end{center}
\end{table}

%% file: Simulation.tex
\section{Simulation study}
We now investigate the power of the procedures from the previous section in a simulation study similar to those described in \cite{Gilbert05} and \cite{Heller2012}. Although our primary focus is on comparing the discrete Benjamini-Yekutieli procedure with its continuous counterpart, we also evaluate the performance of the BH procedure as a classical benchmark as well as the discrete BH and Benjamini-Liu procedures (we omit the continuous Benjamini-Liu procedure).


\subsection{Simulated Scenarios}

We simulate a two-sample problem in which a vector of $m$ independent binary responses (``adverse events") is observed for each subject in the two groups. In order to investigate the influence of discreteness, we consider two cases:
\begin{enumerate}
	\item[a)] High degree of discreteness: Each group consists of $N=25$ subjects.
	\item[b)] Moderate degree of discreteness: Each group consists of $N=100$ subjects.
\end{enumerate}
Then, the goal is to simultaneously test the $m$ null hypotheses $H_{0i}:p_{1i}=p_{2i}, i=1,\ldots,m$, where $p_{1i}$ and $p_{2i}$ are the success probabilities for the $i$th binary response in group 1 and 2 respectively. We take $m=800,2000$ where $m=m_{1}+m_{2}+m_{3}$ and data are generated so that the response is $Bernoulli(0.01)$ at $m_{1}$ positions for both groups, $Bernoulli(0.10)$ at $m_{2}$ positions for both groups and   $Bernoulli(0.10)$ at $m_{3}$ positions for group 1 and  $Bernoulli(q)$ at $m_{3}$ positions for group 2 where $q=0.15,0.25,0.4$ represents weak,  moderate and strong effects respectively. 
The null hypothesis is true for the  $m_{1}$ and $m_{2}$ positions while the alternative hypothesis is true for the  $m_{3}$ positions. We also take different configurations for the proportion of false null hypotheses, $m_{3}$ is set to be $10\%$,  $30\%$ and  $60\%$ of the value of $m$, which represents the complete null hypothesis and small, intermediate and large proportion of effects (the proportion of true nulls $\pi_{0}$ is $1$, $0.9$, $0.7$, $0.4$, respectively). Then, $m_{1}$ is set to be $20\%$,  $50\%$ and  $80\%$ of the true nulls ($m-m_{3}$) and $m_{2}=m-m_{1}-m_{3}$. 

For each of the 54 possible parameter configurations specified by $m,m_{3},m_{1}$ and $q_3$, $10000$ Monte Carlo trials are performed, that is, $10000$ data sets are generated and for each data set, an unadjusted two-sided p-value from Fisher's exact test is computed for each of the $m$ positions, and the multiple testing procedures mentioned above are applied at level $\alpha=0.05$. The power of each procedure was estimated as the fraction of the $m_{3}$ false null hypotheses that were rejected, averaged over the $10000$ simulations. For random number generation the R-function \textit{rbinom} was used. The two-sided p-values from Fisher's exact test were computed using the R-function \textit{fisher.test}.

\subsection{Simulation results}

\subsubsection{High degree of discreteness}
Table \ref{tab:Tab25} in the Appendix displays the (average) power of the seven procedures under investigation.
For weak and moderate effects, i.e. $q_3=0.15$ and $q_3=0.25$, none of the procedure possesses relevant power. For strong effects, the discrete procedures improve considerably on their standard counterparts. Both the discrete and the standard Sarkar procedure are outperformed by the other procedures. For fixed $m_3$ and $q_3$, the power of the discrete Benjamini-Yekutieli procedure is slightly increasing in $m_1$. This may seem surprising, since the $m_1$ positions represent true null hypotheses. The reason for this behavior is that because of the low success probability, these positions are accountable for the `most discrete' part of the data. The more discrete the data are, the flatter the function $G$ (see section \ref{sec:TheoreticalResults} and figure \ref{fig:CompareDiscGContG}) will tend to be, leading to higher critical values for the discrete procedure. 

For strong effects, the results are also displayed in figure \ref{fig:FDPvsModifiedFDPAverageProportionRejections_25}. In order to avoid over-optimism we used, as explained above, for fixed $m_3$ and $q_3$ the configuration with smallest $m_1$. 
\begin{figure}[htbp]
	\centering
		\includegraphics[width=1.00\textwidth]{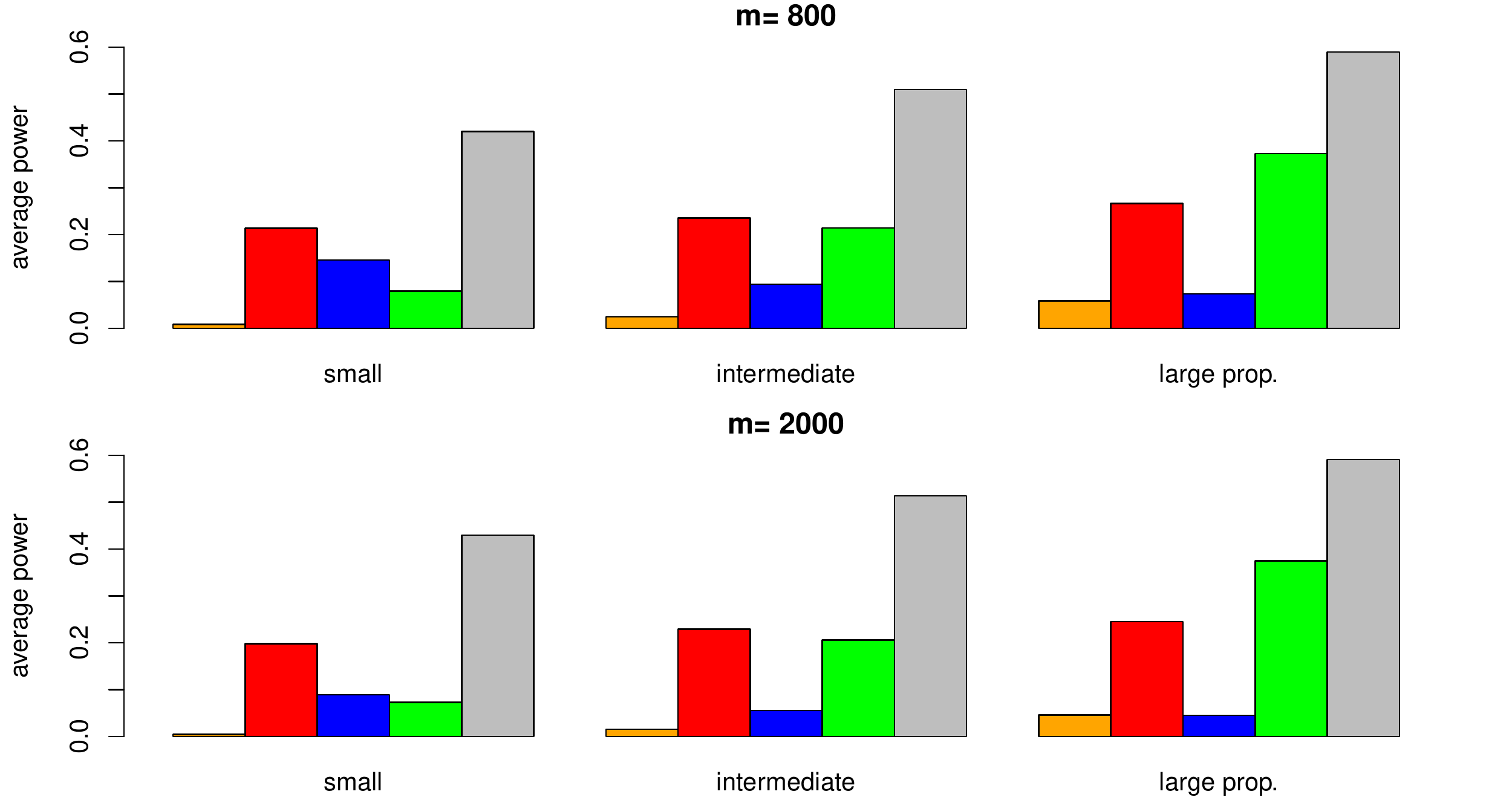}
	\caption{Average power (high degree of discreteness) of the BY-procedure (orange), DBY-procedure (red), DBL-procedure (blue), BH-procedure (green) and DBH (grey) for small, intermediate and large proportions of alternatives with strong effects. The top panel presents results for $m=800$, the lower panel for $m=2000$.}
	\label{fig:FDPvsModifiedFDPAverageProportionRejections_25}
\end{figure} 
Figure \ref{fig:FDPvsModifiedFDPAverageProportionRejections_25} shows that the discrete Benjamini-Yekutieli procedure performs much better than its continuous counterpart. As expected, the discrete Benjamini-Hochberg procedure performs best amongst all procedures. For small and intermediate proportions of large effects, the discrete Benjamini-Yekutieli outperforms the continuous Benjamini-Hochberg procedure, which may be surprising given that the standard Benjamini-Yekutieli procedure has a reputation for being extremely conservative. This shows that by incorporating discreteness,  even the Benjamini-Yekutieli procedure may become a powerful tool in some settings. This observation is similar to the findings of \citet{WestWolf1997} with respect to the Bonferroni procedure.

\subsubsection{Moderate degree of discreteness}
Table \ref{tab:Tab100} in the Appendix displays the (average) power of the seven procedures under investigation, for $N=100$.
For weak effects, again none of the procedure possesses any relevant power. For intermediate and strong effects the situation is different from the highly discrete case. For strong effects, all procedures exhibit roughly the same, extremely high, power. For moderate effects, figure \ref{fig:FDPvsModifiedFDPAverageProportionRejections_100} shows that while the discrete Benjamini-Yekutieli procedure still improves on its continuous counterpart, the gains are now much smaller than in the sparse case.  A similar conclusion holds true for the DBH procedure. 
\begin{figure}[htbp]
	\centering
		\includegraphics[width=1.0\textwidth]{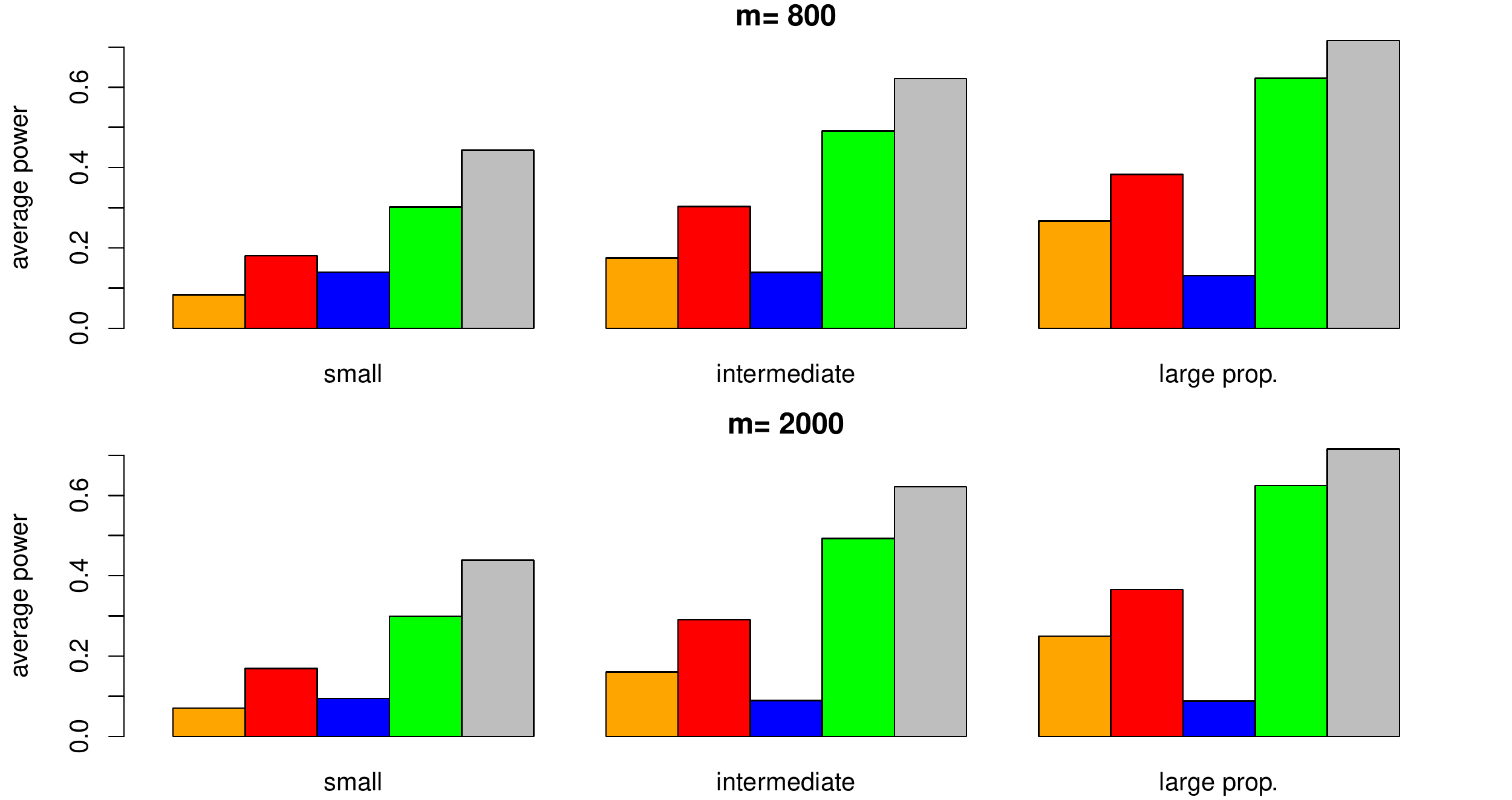}
	\caption{Average power (moderate degree of discreteness) of the BY-procedure (orange), DBY-procedure (red), DBL-procedure (blue), BH-procedure (green) and DBH (grey) for small, intermediate and large proportions of alternatives with intermediate effects. The top panel presents results for $m=800$, the lower panel for $m=2000$.}
	\label{fig:FDPvsModifiedFDPAverageProportionRejections_100}
\end{figure}

%% file: Discussion.tex
\section{Discussion}\label{sec:Discussion}
In this paper we have introduced a discrete modification of the Benjamini-Yekutieli procedure and investigated its performance for an empirical data set and  in a simulation study. We have shown that this procedure can perform remarkably well when the number of tests is large and many variables are sparse. Thus we recommend this procedure for discrete data with arbitrary dependency structure.

The proof of proposition \ref{prop:MainProp} shows that due to \eqref{eq:SarkarBound1} the FDR is actually controlled by
\begin{align*}
\FDR(c) & \le  \sum_{r=1}^n  \frac{1}{r} (G_I(c_r)-G_I(c_{r-1})), 
\end{align*}
where $G_I(x)=\sum_{j \in I} F_j(x)$ and $I \subset \{1, \ldots,n \}$ denotes the set of true hypotheses. Since $n_0=|I|$ is unknown it might be possible to improve our procedures by incorporating an appropriately chosen estimate of $n_0$. Such adaptive procedures have been proposed for the BH procedure by \citet{BenjaminiHochberg00}, \citet{BenjaminiKriegerYekutieli06} and \citet{Storey2004} and others under independence. \citet{Sarkar2008b} investigates such two-stage procedures also for the general dependence case and obtains modifications of the BY procedure. It would be intersting to see if the procedures introduced in this paper can be modified in a similar way and if this leads to more powerful procedures.

The FDR is the expected value of the false discovery proportion $\FDP=V/R$, where $R$ denotes the number of rejected hypotheses and $V$ the number of falsely rejected hypotheses. \citet{GuoHeSarkar2012} generalize the notion of $\FDP$ to the situation where a small number $k$ of false rejections is deemed acceptable. More specifically, they introduce $\kFDP=V/R$ if $V \ge k$ and $0$ otherwise, and the $\kFDR$ is defined as $\erw(\kFDP)$. It can be shown that similar representations to our proposition \ref{prop:MainProp} are available for the $\kFDR$. However, the upper bound in the proposition now involves determining ${n \choose k}$ distribution functions and so this approach seems computationally feasible only for special situations.

Finally, we mention that for independent $p$-values it still remains a challenge to develop a discrete modification of the Benjamini-Hochberg procedure with guaranteed FDR control.

\section*{Acknowledgments} 
The author would like to thank Irene Castro-Conde and Ruth Heller for helpful discussions and Florian Junge for greatly improving the simulation code.


%% file: Appendix.tex
\section*{Appendix 1}
In this appendix we present a numerical example for which the procedure of \citet{Heyse2011} is anticonservative. \HG already constructed a similar example, using the definition \eqref{eq:Def:Adjusted:pvalues} of adjusted p-values. However, as one referee pointed out, this definition differs from the one originally given in \citet{Heyse2011}, which we recap here (in our notation):
\begin{equation}
\begin{aligned} 
\widetilde{pv}_{(n)} &=pv_{(n)}, \qquad \text{and}\\ \label{tab:Heyse:procedure}
\widetilde{pv}_{(i)} &= \min \{\widetilde{pv}_{(i+1)}, G(pv_{(i)})/i\},
\end{aligned}
\end{equation}
for $i=1, \ldots,n-1$. Since $G(pv_{(n)})/n \le pv_{(n)}$, definition \eqref{tab:Heyse:procedure} may yield larger adjusted p-values than definition \eqref{eq:Def:Adjusted:pvalues}. In fact, as the referee pointed out, when the example of \HG is analysed using \eqref{tab:Heyse:procedure} instead of \eqref{eq:Def:Adjusted:pvalues}, the $\FDR$ level is no longer larger than the nominal level.

\subsection*{Numerical example}In what follows we present a numerical example where even the more conservative procedure given by definition \eqref{tab:Heyse:procedure} is anticonservative. Consider $n=4$ $p$-value distributions (under the null hypotheses) defined by
\begin{align*}
P_1 &= 0.05 \cdot \delta_{\{0.05\}} + 0.95 \cdot \delta_{\{1\}},\\
P_2 &= 0.025 \cdot \delta_{\{0.10\}} + 0.975 \cdot \delta_{\{1\}},\\
P_3 &= 0.025 \cdot \delta_{\{0.15\}} + 0.975 \cdot \delta_{\{1\}}\qquad \text{and}\\
P_{4} &= \delta_{\{1\}}.
\end{align*}
where $\delta_{\{x\}}$ denotes the Dirac measure at $x$. Clearly, these distributions are stochastically larger than the uniform distribution and the  function $G=F_1 + F_2+F_3 +F_{4}$ can easily be determined. 
Now consider the configuration where all hypotheses are true and use the discrete FDR procedure to control the $\FDR=\FWER$ at level $\alpha=0.05$ (since $PV_{(4)}=1$  definitions \eqref{eq:Def:Adjusted:pvalues} and \eqref{tab:Heyse:procedure} coincide). This means that all hypotheses with $\widetilde{pv}_{(j)}\le \alpha$ are declared significant. Table \ref{tab:CounterExampleHeyse} illustrates the situation in detail. All possible combinations of $p$-values are listed in the first three columns ($pv_4 =1 $ is omitted). The sorted values are shown in columns 3 to 6, the adjusted $p$-values are given in columns 7 to 9. Columns 10 to 12 list the probability masses for the $p$-values in the first 3 columns. Since we assume independence, the joint probability of observing $(pv_1, pv_2, pv_3)$ is given by the product of the probabilities (last column) and the joint probabilities for those constellations of $p$-values for which at least one hypothtesis is rejected are printed in boldface. Thus the probability of at least one rejection can be computed exactly and we obtain $\FWER=\FDR=0.05059375$.

In principle this type of example can be extended easily to larger $n$, however since $2^{n-1}$ events have to be enumerated we soon run into computational difficulties. As a slight extension of the example above consider $n=10$. Assume $P_1$ as above, $P_{10}= \delta_{\{1\}}$ and  $P_2,\ldots, P_9$ each with probability mass  $0.00621$ at locations $0.10, \ldots , 0.45$. In this case (the details are not presented) similar calculations as above yield  $\FWER=\FDR=0.05100062$.

\begin{sidewaystable}[htbp]
  \centering
  \caption{Details of the example}
    \begin{tabular}{ccccccccccccc}
    \toprule
    $pv_1$     & $pv_2$     & $pv_3$     & $pv_{(1)}$ & $pv_{(2)}$ & $pv_{(3)}$ & $\widetilde{pv}_{(1)}$ & $\widetilde{pv}_{(2)}$ & $\widetilde{pv}_{(3)}$ & $P(PV_1=pv_1)$ & $P(PV_2=pv_2)$ & $P(PV_3=pv_3)$ & joint prob.  \\
    \midrule
		0.0500 & 0.1000 & 0.1500 & 0.0500 & 0.1000 & 0.1500 & 0.0333 & 0.0333 & 0.0333 & 0.0500 & 0.0250 & 0.0250 & \textbf{0.00003125} \\
    1.0000 & 0.1000 & 0.1500 & 0.1000 & 0.1500 & 1.0000 & 0.0500 & 0.0500 & 1.0000 & 0.9500 & 0.0250 & 0.0250 & \textbf{0.00059375} \\
    0.0500 & 1.0000 & 0.1500 & 0.0500 & 0.1500 & 1.0000 & 0.0500 & 0.0500 & 1.0000 & 0.0500 & 0.9750 & 0.0250 & \textbf{0.00121875} \\
    1.0000 & 1.0000 & 0.1500 & 0.1500 & 1.0000 & 1.0000 & 0.1000 & 1.0000 & 1.0000 & 0.9500 & 0.9750 & 0.0250 & 0.02315625 \\
    0.0500 & 0.1000 & 1.0000 & 0.0500 & 0.1000 & 1.0000 & 0.0375 & 0.0375 & 1.0000 & 0.0500 & 0.0250 & 0.9750 & \textbf{0.00121875} \\
    1.0000 & 0.1000 & 1.0000 & 0.1000 & 1.0000 & 1.0000 & 0.0750 & 1.0000 & 1.0000 & 0.9500 & 0.0250 & 0.9750 & 0.02315625 \\
    0.0500 & 1.0000 & 1.0000 & 0.0500 & 1.0000 & 1.0000 & 0.0500 & 1.0000 & 1.0000 & 0.0500 & 0.9750 & 0.9750 & \textbf{0.04753125} \\
    1.0000 & 1.0000 & 1.0000 & 1.0000 & 1.0000 & 1.0000 & 1.0000 & 1.0000 & 1.0000 & 0.9500 & 0.9750 & 0.9750 & 0.90309375 \\\hline
          &       &       &       &       &       &       &       &       &       &       &      FWER  &        \textbf{0.05059375}  \\
    \bottomrule
    \end{tabular}%
  \label{tab:CounterExampleHeyse}%
\end{sidewaystable}%

\subsection*{Mathematical considerations}The above example implies that the reasoning provided in \citet{Heyse2011} is inadequate for proving FDR control of this procedure. We now try to identify which part of this reasoning may be problematic. \citet{Heyse2011} argued that this procedure controls the $\FDR$ by the same argument used in \citet{Gilbert05}. The latter paper appealed to equation (19) of \citet{BenjaminiYekutieli01}, which we reproduce here (with our notation): For a given sequence $c_1 \le \cdots \le c_n$ of critical constants the $\FDR$ can be expressed by
\begin{align}
\FDR &=\sum_{j\in I} \sum_{r=1}^n \frac{1}{r} F_j (c_r) \cdot P(C_r^{(j)})
\end{align} 
where the event $C_r^{(j)}$ is defined on page 1178 of \citet{BenjaminiYekutieli01}. The crucial point is that the events $C_r^{(j)}$ and  probabilities $P(C_r^{(j)})$ are \emph{determined} (in a non-trivial way) by the given multiple testing procedure and the distribution of the $p$-values. In \citet{Gilbert05}, however, the values $P(C_r^{(j)})$ are  \emph{chosen} (arbitrarily) without further mathematical justification. The author seems to have been aware that this reasoning does not constitute a valid mathematical proof of FDR control since he states that this procedure 'usually' (p. 151) controls the FDR. Thus, while this procedure has some heuristic justification, its exact mathematical properties remain unclear. 

\subsection*{Discussion} We have shown that the DBH procedure of \citet{Heyse2011} does not generally guarantee FDR control at the desired level. We have also tried to identify the mathematical reason for this. It might be argued that the magnitude of the bias in the example(s) is rather small. However, it should be kept in mind that these examples show the behaviour for specific constellations and it is not by any means clear that these are 'worst case scenarios' (which is what type 1 error control is all about). With more mathematical and computational effort it may also be possible to find more severe examples of FDR inflation (e.g. it is unclear what may happen when $m$ is large, some hypotheses are false, etc). Thus our example is not meant  to quantify the magnitude of bias, but rather to demonstrate that the DBH procedure does not rigorously control FDR. To the best of our knowledge, no (valid) upper FDR bound is available for the DBH procedure. Thus it still remains an open problem to find a discrete modification of the BH procedure that rigorously controls the FDR.

\clearpage

\section*{Appendix 2}

\begin{table}[htbp]
  \centering
  \caption{Average power of FDR procedures for $N=25$}
    \begin{tabularx}{1.1\textwidth}{cccc||ccccccc}\toprule		
    $m$     & $m_3$    & $m_1$    & $q_3$    & DBH & BH    & DBL   & DBY   & BY    & DSarkar & Sarkar \\ 
		\midrule
		800   & 80    & 144   & 0.15  & 0.0002 & 0.0000 & 0.0004 & 0.0002 & 0.0000 & 0.0000 & 0.0000 \\
          &       & 144   & 0.25  & 0.0157 & 0.0003 & 0.0084 & 0.0080 & 0.0000 & 0.0001 & 0.0000 \\
          &       & 144   & 0.4   & 0.4195 & 0.0795 & 0.1457 & 0.2131 & 0.0086 & 0.0100 & 0.0001 \\
          &       & 360   & 0.15  & 0.0005 & 0.0000 & 0.0007 & 0.0003 & 0.0000 & 0.0000 & 0.0000 \\
          &       & 360   & 0.25  & 0.0188 & 0.0003 & 0.0096 & 0.0087 & 0.0000 & 0.0001 & 0.0000 \\
          &       & 360   & 0.4   & 0.4511 & 0.0795 & 0.1466 & 0.2232 & 0.0086 & 0.0101 & 0.0001 \\
          &       & 576   & 0.15  & 0.0006 & 0.0000 & 0.0007 & 0.0006 & 0.0000 & 0.0000 & 0.0000 \\
          &       & 576   & 0.25  & 0.0233 & 0.0003 & 0.0152 & 0.0097 & 0.0000 & 0.0001 & 0.0000 \\
          &       & 576   & 0.4   & 0.4789 & 0.0794 & 0.1551 & 0.2416 & 0.0086 & 0.0103 & 0.0001 \\
          & 240   & 112   & 0.15  & 0.0002 & 0.0000 & 0.0002 & 0.0002 & 0.0000 & 0.0000 & 0.0000 \\
          &       & 112   & 0.25  & 0.0219 & 0.0005 & 0.0081 & 0.0062 & 0.0000 & 0.0000 & 0.0000 \\
          &       & 112   & 0.4   & 0.5095 & 0.2142 & 0.0942 & 0.2350 & 0.0241 & 0.0212 & 0.0001 \\
          &       & 280   & 0.15  & 0.0002 & 0.0000 & 0.0002 & 0.0002 & 0.0000 & 0.0000 & 0.0000 \\
          &       & 280   & 0.25  & 0.0242 & 0.0005 & 0.0082 & 0.0064 & 0.0000 & 0.0000 & 0.0000 \\
          &       & 280   & 0.4   & 0.5528 & 0.2141 & 0.0955 & 0.2494 & 0.0241 & 0.0217 & 0.0001 \\
          &       & 448   & 0.15  & 0.0003 & 0.0000 & 0.0005 & 0.0002 & 0.0000 & 0.0000 & 0.0000 \\
          &       & 448   & 0.25  & 0.0279 & 0.0005 & 0.0082 & 0.0067 & 0.0000 & 0.0000 & 0.0000 \\
          &       & 448   & 0.4   & 0.5875 & 0.2140 & 0.0961 & 0.2650 & 0.0241 & 0.0223 & 0.0001 \\
          & 480   & 64    & 0.15  & 0.0002 & 0.0000 & 0.0002 & 0.0002 & 0.0000 & 0.0000 & 0.0000 \\
          &       & 64    & 0.25  & 0.0279 & 0.0007 & 0.0031 & 0.0053 & 0.0000 & 0.0000 & 0.0000 \\
          &       & 64    & 0.4   & 0.5895 & 0.3727 & 0.0732 & 0.2664 & 0.0586 & 0.1181 & 0.0001 \\
          &       & 160   & 0.15  & 0.0002 & 0.0000 & 0.0002 & 0.0002 & 0.0000 & 0.0000 & 0.0000 \\
          &       & 160   & 0.25  & 0.0290 & 0.0007 & 0.0031 & 0.0054 & 0.0000 & 0.0000 & 0.0000 \\
          &       & 160   & 0.4   & 0.5903 & 0.3725 & 0.0732 & 0.2687 & 0.0586 & 0.1221 & 0.0001 \\
          &       & 256   & 0.15  & 0.0002 & 0.0000 & 0.0002 & 0.0002 & 0.0000 & 0.0000 & 0.0000 \\
          &       & 256   & 0.25  & 0.0306 & 0.0007 & 0.0032 & 0.0055 & 0.0000 & 0.0000 & 0.0000 \\
          &       & 256   & 0.4   & 0.6064 & 0.3722 & 0.0733 & 0.2697 & 0.0586 & 0.1267 & 0.0001 \\
    2000  & 200   & 360   & 0.15  & 0.0002 & 0.0000 & 0.0002 & 0.0001 & 0.0000 & 0.0000 & 0.0000 \\
          &       & 360   & 0.25  & 0.0127 & 0.0001 & 0.0075 & 0.0048 & 0.0000 & 0.0000 & 0.0000 \\
          &       & 360   & 0.4   & 0.4293 & 0.0726 & 0.0887 & 0.1981 & 0.0046 & 0.0036 & 0.0000 \\
          &       & 900   & 0.15  & 0.0002 & 0.0000 & 0.0002 & 0.0002 & 0.0000 & 0.0000 & 0.0000 \\
          &       & 900   & 0.25  & 0.0149 & 0.0001 & 0.0081 & 0.0052 & 0.0000 & 0.0000 & 0.0000 \\
          &       & 900   & 0.4   & 0.4520 & 0.0725 & 0.0949 & 0.2197 & 0.0046 & 0.0036 & 0.0000 \\
          &       & 1440  & 0.15  & 0.0002 & 0.0000 & 0.0003 & 0.0002 & 0.0000 & 0.0000 & 0.0000 \\
          &       & 1440  & 0.25  & 0.0188 & 0.0001 & 0.0082 & 0.0056 & 0.0000 & 0.0000 & 0.0000 \\
          &       & 1440  & 0.4   & 0.4700 & 0.0725 & 0.0975 & 0.2291 & 0.0046 & 0.0036 & 0.0000 \\
          & 600   & 280   & 0.15  & 0.0002 & 0.0000 & 0.0002 & 0.0001 & 0.0000 & 0.0000 & 0.0000 \\
          &       & 280   & 0.25  & 0.0186 & 0.0001 & 0.0030 & 0.0042 & 0.0000 & 0.0000 & 0.0000 \\
          &       & 280   & 0.4   & 0.5135 & 0.2056 & 0.0555 & 0.2287 & 0.0155 & 0.0072 & 0.0000 \\
          &       & 700   & 0.15  & 0.0002 & 0.0000 & 0.0002 & 0.0001 & 0.0000 & 0.0000 & 0.0000 \\
          &       & 700   & 0.25  & 0.0214 & 0.0001 & 0.0030 & 0.0043 & 0.0000 & 0.0000 & 0.0000 \\
          &       & 700   & 0.4   & 0.5539 & 0.2055 & 0.0556 & 0.2302 & 0.0155 & 0.0073 & 0.0000 \\
          &       & 1120  & 0.15  & 0.0002 & 0.0000 & 0.0002 & 0.0001 & 0.0000 & 0.0000 & 0.0000 \\
          &       & 1120  & 0.25  & 0.0265 & 0.0001 & 0.0030 & 0.0045 & 0.0000 & 0.0000 & 0.0000 \\
          &       & 1120  & 0.4   & 0.5897 & 0.2055 & 0.0558 & 0.2497 & 0.0155 & 0.0075 & 0.0000 \\
          & 1200  & 160   & 0.15  & 0.0001 & 0.0000 & 0.0001 & 0.0001 & 0.0000 & 0.0000 & 0.0000 \\
          &       & 160   & 0.25  & 0.0270 & 0.0002 & 0.0030 & 0.0040 & 0.0000 & 0.0000 & 0.0000 \\
          &       & 160   & 0.4   & 0.5907 & 0.3749 & 0.0451 & 0.2449 & 0.0453 & 0.1026 & 0.0000 \\
          &       & 400   & 0.15  & 0.0001 & 0.0000 & 0.0001 & 0.0001 & 0.0000 & 0.0000 & 0.0000 \\
          &       & 400   & 0.25  & 0.0288 & 0.0002 & 0.0030 & 0.0041 & 0.0000 & 0.0000 & 0.0000 \\
          &       & 400   & 0.4   & 0.5908 & 0.3745 & 0.0451 & 0.2573 & 0.0453 & 0.1067 & 0.0000 \\
          &       & 640   & 0.15  & 0.0001 & 0.0000 & 0.0002 & 0.0001 & 0.0000 & 0.0000 & 0.0000 \\
          &       & 640   & 0.25  & 0.0309 & 0.0002 & 0.0030 & 0.0041 & 0.0000 & 0.0000 & 0.0000 \\
          &       & 640   & 0.4   & 0.6059 & 0.3742 & 0.0451 & 0.2662 & 0.0453 & 0.1118 & 0.0000 \\
    \bottomrule
    \end{tabularx}%
  \label{tab:Tab25}%
\end{table}%

\newpage
\begin{table}[htbp]
  \centering
  \caption{Average power of FDR procedures for $N=100$}
    \begin{tabularx}{1.1\textwidth}{cccc||ccccccc}
    \toprule
    $m$     &$ m_3$    & $m_1$    & $q_3$    & DBH & BH    & DBL   & DBY   & BY    & DSarkar & Sarkar \\
    \midrule
		800   & 80    & 144   & 0.15  & 0.0022 & 0.0010 & 0.0019 & 0.0005 & 0.0002 & 0.0000 & 0.0000 \\
          &       & 144   & 0.25  & 0.4429 & 0.3013 & 0.1396 & 0.1806 & 0.0829 & 0.0144 & 0.0033 \\
          &       & 144   & 0.4   & 0.9931 & 0.9848 & 0.8955 & 0.9646 & 0.9396 & 0.9413 & 0.8897 \\
          &       & 360   & 0.15  & 0.0027 & 0.0010 & 0.0028 & 0.0006 & 0.0002 & 0.0000 & 0.0000 \\
          &       & 360   & 0.25  & 0.5114 & 0.3005 & 0.1656 & 0.2205 & 0.0828 & 0.0191 & 0.0033 \\
          &       & 360   & 0.4   & 0.9950 & 0.9847 & 0.9204 & 0.9748 & 0.9396 & 0.9460 & 0.8896 \\
          &       & 576   & 0.15  & 0.0049 & 0.0010 & 0.0043 & 0.0010 & 0.0002 & 0.0000 & 0.0000 \\
          &       & 576   & 0.25  & 0.6095 & 0.2996 & 0.2086 & 0.3085 & 0.0828 & 0.0344 & 0.0033 \\
          &       & 576   & 0.4   & 0.9973 & 0.9847 & 0.9451 & 0.9842 & 0.9396 & 0.9652 & 0.8896 \\
          & 240   & 112   & 0.15  & 0.0026 & 0.0011 & 0.0018 & 0.0004 & 0.0002 & 0.0000 & 0.0000 \\
          &       & 112   & 0.25  & 0.6213 & 0.4915 & 0.1391 & 0.3030 & 0.1751 & 0.0776 & 0.0094 \\
          &       & 112   & 0.4   & 0.9975 & 0.9945 & 0.9203 & 0.9840 & 0.9690 & 0.9847 & 0.9712 \\
          &       & 280   & 0.15  & 0.0035 & 0.0011 & 0.0022 & 0.0005 & 0.0002 & 0.0000 & 0.0000 \\
          &       & 280   & 0.25  & 0.6593 & 0.4907 & 0.1458 & 0.3429 & 0.1750 & 0.1131 & 0.0094 \\
          &       & 280   & 0.4   & 0.9980 & 0.9944 & 0.9384 & 0.9858 & 0.9690 & 0.9872 & 0.9712 \\
          &       & 448   & 0.15  & 0.0057 & 0.0011 & 0.0029 & 0.0007 & 0.0002 & 0.0000 & 0.0000 \\
          &       & 448   & 0.25  & 0.7133 & 0.4897 & 0.1773 & 0.3933 & 0.1749 & 0.1776 & 0.0094 \\
          &       & 448   & 0.4   & 0.9988 & 0.9944 & 0.9621 & 0.9901 & 0.9690 & 0.9908 & 0.9712 \\
          & 480   & 64    & 0.15  & 0.0038 & 0.0013 & 0.0016 & 0.0004 & 0.0002 & 0.0000 & 0.0000 \\
          &       & 64    & 0.25  & 0.7163 & 0.6227 & 0.1308 & 0.3832 & 0.2672 & 0.3458 & 0.1111 \\
          &       & 64    & 0.4   & 0.9988 & 0.9976 & 0.9488 & 0.9898 & 0.9825 & 0.9951 & 0.9910 \\
          &       & 160   & 0.15  & 0.0044 & 0.0013 & 0.0018 & 0.0005 & 0.0002 & 0.0000 & 0.0000 \\
          &       & 160   & 0.25  & 0.7312 & 0.6221 & 0.1387 & 0.4114 & 0.2672 & 0.3790 & 0.1110 \\
          &       & 160   & 0.4   & 0.9989 & 0.9976 & 0.9643 & 0.9909 & 0.9825 & 0.9957 & 0.9910 \\
          &       & 256   & 0.15  & 0.0057 & 0.0013 & 0.0019 & 0.0006 & 0.0002 & 0.0000 & 0.0000 \\
          &       & 256   & 0.25  & 0.7539 & 0.6215 & 0.1448 & 0.4285 & 0.2671 & 0.4179 & 0.1110 \\
          &       & 256   & 0.4   & 0.9991 & 0.9976 & 0.9794 & 0.9916 & 0.9825 & 0.9963 & 0.9910 \\
    2000  & 200   & 360   & 0.15  & 0.0011 & 0.0003 & 0.0009 & 0.0002 & 0.0000 & 0.0000 & 0.0000 \\
          &       & 360   & 0.25  & 0.4390 & 0.2994 & 0.0948 & 0.1688 & 0.0707 & 0.0063 & 0.0009 \\
          &       & 360   & 0.4   & 0.9930 & 0.9848 & 0.8520 & 0.9646 & 0.9336 & 0.9409 & 0.8886 \\
          &       & 900   & 0.15  & 0.0015 & 0.0003 & 0.0013 & 0.0003 & 0.0000 & 0.0000 & 0.0000 \\
          &       & 900   & 0.25  & 0.5116 & 0.2986 & 0.1088 & 0.2066 & 0.0706 & 0.0083 & 0.0009 \\
          &       & 900   & 0.4   & 0.9949 & 0.9848 & 0.8793 & 0.9749 & 0.9336 & 0.9454 & 0.8886 \\
          &       & 1440  & 0.15  & 0.0030 & 0.0003 & 0.0020 & 0.0005 & 0.0000 & 0.0000 & 0.0000 \\
          &       & 1440  & 0.25  & 0.6091 & 0.2977 & 0.1433 & 0.2926 & 0.0706 & 0.0171 & 0.0009 \\
          &       & 1440  & 0.4   & 0.9973 & 0.9847 & 0.9208 & 0.9830 & 0.9336 & 0.9648 & 0.8886 \\
          & 600   & 280   & 0.15  & 0.0015 & 0.0004 & 0.0009 & 0.0002 & 0.0001 & 0.0000 & 0.0000 \\
          &       & 280   & 0.25  & 0.6212 & 0.4928 & 0.0895 & 0.2902 & 0.1604 & 0.0601 & 0.0028 \\
          &       & 280   & 0.4   & 0.9976 & 0.9946 & 0.8775 & 0.9828 & 0.9680 & 0.9848 & 0.9716 \\
          &       & 700   & 0.15  & 0.0020 & 0.0004 & 0.0010 & 0.0002 & 0.0001 & 0.0000 & 0.0000 \\
          &       & 700   & 0.25  & 0.6564 & 0.4920 & 0.0988 & 0.3198 & 0.1603 & 0.0970 & 0.0028 \\
          &       & 700   & 0.4   & 0.9981 & 0.9945 & 0.8954 & 0.9855 & 0.9680 & 0.9872 & 0.9716 \\
          &       & 1120  & 0.15  & 0.0034 & 0.0004 & 0.0015 & 0.0003 & 0.0001 & 0.0000 & 0.0000 \\
          &       & 1120  & 0.25  & 0.7125 & 0.4910 & 0.1189 & 0.3688 & 0.1603 & 0.1672 & 0.0028 \\
          &       & 1120  & 0.4   & 0.9988 & 0.9945 & 0.9270 & 0.9886 & 0.9680 & 0.9909 & 0.9716 \\
          & 1200  & 160   & 0.15  & 0.0024 & 0.0006 & 0.0008 & 0.0002 & 0.0001 & 0.0000 & 0.0000 \\
          &       & 160   & 0.25  & 0.7151 & 0.6244 & 0.0882 & 0.3654 & 0.2494 & 0.3449 & 0.0958 \\
          &       & 160   & 0.4   & 0.9988 & 0.9976 & 0.9204 & 0.9885 & 0.9803 & 0.9951 & 0.9910 \\
          &       & 400   & 0.15  & 0.0029 & 0.0006 & 0.0009 & 0.0002 & 0.0001 & 0.0000 & 0.0000 \\
          &       & 400   & 0.25  & 0.7294 & 0.6238 & 0.0914 & 0.3860 & 0.2494 & 0.3776 & 0.0957 \\
          &       & 400   & 0.4   & 0.9989 & 0.9976 & 0.9382 & 0.9898 & 0.9803 & 0.9957 & 0.9910 \\
          &       & 640   & 0.15  & 0.0037 & 0.0006 & 0.0009 & 0.0003 & 0.0001 & 0.0000 & 0.0000 \\
          &       & 640   & 0.25  & 0.7531 & 0.6232 & 0.0976 & 0.4166 & 0.2493 & 0.4197 & 0.0956 \\
          &       & 640   & 0.4   & 0.9990 & 0.9976 & 0.9615 & 0.9910 & 0.9803 & 0.9963 & 0.9910 \\
    \bottomrule
    \end{tabularx}%
  \label{tab:Tab100}%
\end{table}%